\documentclass[11pt,oneside,reqno]{amsart}

\usepackage{graphicx}
\usepackage{tikz-cd}
\usepackage{mathtools}
\usepackage{amsfonts}
\usepackage{amscd,amsmath}
\usepackage{amssymb}
\usepackage{setspace}
\usepackage{blindtext}
\usepackage{graphicx}
\usepackage{xfrac}
\usepackage{amsmath}
\usepackage[normalem]{ulem}

\usepackage{gastex}
\usepackage{longtable}
\usepackage{lscape}
\usepackage{tabularx}
\usepackage{multicol}
\usepackage{verbatim}
\usepackage{multirow}
\usepackage{epsfig, graphics,graphicx}
\usepackage{geometry}
\usepackage{amscd}
\usepackage{xcolor}
\usepackage{hyperref}
\usepackage[utf8]{inputenc}

\usepackage{amssymb,amsmath,amsthm, verbatim, calc}
\numberwithin{equation}{section}
\usepackage[english]{babel}

\definecolor{armygreen}{rgb}{0.29, 0.33, 0.13}
\definecolor{darkgreen}{rgb}{0.0, 0.2, 0.13}

\newtheorem{thm}{Theorem}[section]
\newtheorem{lem}[thm]{Lemma}

\newtheorem{Def}[thm]{Definition}
\newtheorem{prop}[thm]{Proposition}
\newtheorem{rem}[thm]{Remark}

\newcommand{\Rem}{\begin{rem} \rm}
\newcommand{\bdfn}{\begin{Def} \rm}
\newcommand{\edfn}{\end{Def}}

\newcommand{\ba}{\begin{array}}
\newcommand{\ea}{\end{array}}

\begin{document}
\title[Structure of Generalized bi-circular idempotents and isometric reflections]{Structure of Generalized bi-circular idempotents and isometric reflections on $C^1[0,1]$}
		
\author{Himanshu Kumar}
\address{Department of Applied Sciences, Indian Institute of Information Technology Allahabad, Prayagraj-211015, U.P., India.}
\email{himanshumath1507@gmail.com}

\author{Hemant Kumar}
\address{Department of Applied Sciences, Indian Institute of Information Technology Allahabad, Prayagraj-211015, U.P., India.}
\email{hemantkr.math@gmail.com} 

\author{Abdullah Bin Abu Baker}
\address{Department of Applied Sciences, Indian Institute of Information Technology Allahabad, Prayagraj-211015, U.P., India.}
\email{abdullahmath@gmail.com}

		

\thanks{The first-named and second-named authors gratefully acknowledge the Ministry of Education, New Delhi (India), for financial support, and IIIT Allahabad, India, for providing the resources and infrastructure to carry out this research. The third-named author is partially supported by Anusandhan National Research Foundation (MATRICS) grant No. MTR/2022/000710.}
		
\date{\today}
		
\begin{abstract} 
Let $C^1[0,1]$ be the complex linear space of all continuously differentiable complex-valued functions $f$ on the unit interval $[0,1]$ with respect to the norm $ \|f\|_{\sigma} = |f(0)| +\|f'\|_\infty$. Let $P_1, P_2: C^1[0,1] \rightarrow C^1[0,1]$ be distinct, nonzero idempotent maps, which are not necessarily linear, such that $P_1P_2 = P_2P_1 = 0$ and $P_1+P_2 = I$, where $I$ denotes the identity operator. It is proved that $\lambda_1P_1 + \lambda_2P_2$ is an isometry on $C^1[0,1]$, for some distinct unit modulus complex numbers $\lambda_1, \lambda_2$, if and only if either $\lambda_1 + \lambda_2 = 0$, or $\lambda_1P_1 + \lambda_2P_2$ is an isometry for all such complex numbers $\lambda_1, \lambda_2$. In the former case, the collection $\{P_1, P_2\}$ is called a family of generalized bi-circular idempotents; in the latter case, it is called a family of bi-circular idempotents. The structure of isometric reflection on $C^1[0,1]$, that is, an isometry $T$ such that $T^2 = I$, is characterized, and its relationship with the above class of idempotents maps is also discussed. 
\end{abstract}
		
\maketitle
		
\thispagestyle{empty}

\section{Introduction and Preliminary Results}		

The study of bi-circular idempotents and their generalizations started recently by Botelho, Miura, Abu Baker, and Kumar \cite{FT1, FT2, HAF1, HAF2}. A basic problem is to determine the structure of these mappings on a given complex normed linear space. It is often happens that such maps are related to isometric reflections on a given space. 

Let $X$ be a complex normed space. A map $P: X \rightarrow X$ is called idempotent if $P^2 = P$. If $P$ is linear and bounded, it is usually called a projection. A map $T: X \rightarrow X$ is called an isometry if $||Tx - Ty|| = ||x  -y||$ for all $x,y \in X$. It is called an isometric reflection if $T^2 = I$. Consider a collection $\mathcal{C} = \{P_1,P_2\}$ of distinct nonzero idempotent maps on $X$ such that $P_1P_2 = P_2P_1 = 0$ and $P_1+P_2 = I$. We say that $\mathcal{C}$ is a family of {\bf \em generalized bi-circular idempotents} (in short, $GBCI$) if $T = \lambda_1P_1 + \lambda_2P_2$ is an isometry on $X$, for some distinct unit modulus complex numbers $\lambda_1, \lambda_2$. It is called a family of {\bf \em bi-circular idempotents} (in short, $BCI$) if $\lambda_1P_1 + \lambda_2P_2$ is an isometry on $X$, for all unit modulus complex numbers $\lambda_1, \lambda_2$. In the former case, each of the maps $P_1,P_2$ is called generalized bi-circular idempotent; in the latter case, they are called bi-circular idempotents. We refer to $T$ as the isometry associated with the family $\mathcal{C}$. We emphasize here that in all maps under consideration, we don't assume any linearity. The linear version of the above problem is investigated extensively, and we refer interested readers to the papers \cite{FJ, FPD, MDC, DNN, JJ, LB1, LB2} and references therein. 

In \cite{FT1, FT2}, Botelho and Miura characterized $GBCI$ on $C^1[0,1]$ with the norm determined by a compact and connected subset $D$ of $[0,1]^2$ such that $\pi_1(D) \cup \pi_2(D) =[0,1]$, where $\pi_j$ is the projection map from $D$ to $j$-th coordinate of $[0,1]^2$. Later, Kumar, Abu Baker and Botelho \cite{HAF1} characterized $GBCI$ on the space $\mathcal{S}_A$ of analytic functions on the open unit disc $\mathbb{D}$ whose derivative can be extended to the closed unit disc $\overline{\mathbb{D}}$, and the space $\mathcal{S}^\infty$ of analytic functions on $\mathbb{D}$ with bounded derivatives. It was proved in \cite{HAF1} that the isometry $T$, associated with a family $\{P_1, P_2\}$ of $GBCI$, as well as $P_1$ and $P_2$ are real linear. Moreover, if $T =\lambda_1 P_1 + \lambda_2 P_2$, then $TP_1 = \lambda_1 P_1$  and $TP_2 = \lambda_2 P_2$. We record these two results for future reference. 

\begin{prop} \cite{HAF1} \label{T0=0}
Let $\mathcal{C} = \{P_1, P_2 \}$ be a family of generalized bi-circular idempotents corresponding to a surjective isometry $T$. Then $T(0) = 0$. Moreover, $T$ is a real linear.
\end{prop}

\begin{lem} \cite{HAF1} \label{TP=lambdaP}
Let $\mathcal{C} = \{P_1, P_2\}$ be a family of generalized bi-circular idempotents such that $T =\lambda_1 P_1 + \lambda_2 P_2$. Then $TP_1 = \lambda_1 P_1$  and $TP_2 = \lambda_2 P_2$.
\end{lem}

The notion of generalized $k$-circular idempotents was introduced by Kumar, Abu Baker, and Botelho in \cite{HAF2}. In the same paper, the authors also introduced several new classes of idempotent maps and studied their properties. 

Motivated by these works, in this paper, we completely describe the structure of $GBCI$ on $C^1[0,1]$, the space of all continuously differentiable complex-valued functions on the unit interval $[0,1]$ with respect to the norm $ \|f\|_{\sigma} = |f(0)| +\|f'\|_\infty$. Our first result shows that any $GBCI$ is Lipschitz continuous.

\begin{prop}
Let $\mathcal{C} = \{P_1, P_2\}$ be a collection of generalized bi-circular idempotents on a normed space $X$. Then $P_1, P_2$ are Lipschitz continuous.
\end{prop}

\begin{proof} Suppose that $ \mathcal{C} = \{P_1, P_2\}$ is a collection of generalized bi-circular idempotents corresponding to a surjective isometry $T$ on $X$. Then $T = \lambda_1P_1 +  \lambda_2P_2$ for some distinct $\lambda_1$, $\lambda_2 \in \mathbb{T}$. Now, for any $x,y \in X$, and $i,j = 1, 2$, $i \neq j$, we have
\begin{align*}
\|P_i x-P_i y\| &= \Bigg\|\Bigg[\frac{Tx - \lambda_j x}{\lambda_i - \lambda_j}\Bigg]- \Bigg[\frac {Ty - \lambda_j y}{\lambda_i - \lambda_j}\Bigg] \Bigg\|  \\
& = \frac{1}{|\lambda_i - \lambda_j|} \Big\|[Tx-Ty] - \lambda_j [x - y] \Big\| \\
& \leq \frac{1}{|\lambda_i - \lambda_j|} \Big[\|x-y\| + |\lambda_j| \|x - y \| \Big] \\
& = \frac{2}{|\lambda_i - \lambda_j|} \ \|x-y\|.
\end{align*}
Therefore, $P_i$, $i = 1,2$, is Lipschitz continuous. 
\end{proof} 

The structure of surjective isometries on space $C^1[0,1]$, equipped with norm $\|f \|_{\sigma} = |f(0)| + \|f'\|_\infty $ is given in the following result by Miura. 
		
\begin{thm} \cite{M}
Let $T: C^1[0,1] \to C^1[0,1]$ be a surjective isometry, which need not be linear, with respect to the norm $\|f \|_{\sigma} = |f(0)| + \|f'\|_\infty $. Then there exists a constant $c\in \mathbb{C}$ with $|c| =1$, a continuous unimodular function $\beta : [0,1] \to \mathbb{C}$ and a homeomorphism $\phi: [0,1] \to [0,1]$ such that 
\begin{equation*}
T(f)(t) = T(0)(t) + cf(0) + \int_{0}^{t} \beta(s) f'(\phi(s)) ds, \tag*{(Form I)} \text{ or }
\end{equation*}
\begin{equation*}
T(f)(t) = T(0)(t) + c\overline{f(0)} + \int_{0}^{t} \beta(s) f'(\phi(s)) ds, \tag*{(Form II)} \text{ or }
\end{equation*}
\begin{equation*}
T(f)(t) = T(0)(t) + cf(0) + \int_{0}^{t} \beta(s) \overline{ f'(\phi(s))}ds, \tag*{(Form III)} \text{ or }
\end{equation*}
\begin{equation*}
T(f)(t) = T(0)(t) + c\overline{f(0)} + \int_{0}^{t} \beta(s) \overline{ f'(\phi(s))} ds , \tag*{(Form IV)}
\end{equation*}
for every $f \in C^1 [0,1]$ and $t \in [0,1]$. Conversely, each of the above maps is a surjective isometry on $C^1[0,1]$ with respect to the norm $\|.\|_{\sigma}$, where $T(0)$ is an arbitrary element of $C^1[0,1]$.
\end{thm}		

\section{Isometric reflection}

In this section, we give characterizations of isometric reflections on $C^1[0,1]$.

\begin{thm}
Let $T$ be an isometric reflection on $C^1[0,1]$ of Form $(I)$. Then $c = \pm 1$ with $\beta(t) \beta(\phi (t)) =1$ and $\phi^2(t) =t$ for every $t \in [0,1]$.
\end{thm}
\begin{proof}
An isometry on $C^1[0,1]$ of Form $(I)$ is described by	
\begin{equation*} \label{Form1}
T(f)(t)= T(0)(t)+  cf(0) + \int_{0}^{t} \beta(s)f'(\phi(s))ds
\end{equation*}
for every $f \in C^1[0,1]$ and $t \in [0,1]$. Consequently, we obtain 
\begin{align*}
T(f)(0) = T(0)(0) + cf(0)   \text{ and }  (T(f))'(t)= (T(0))'(t)+ \beta(t) f'(\phi(t)).
\end{align*}
If $T$ is an isometric reflection, then $T^2(f)(t)= f(t)$ for every $f \in C^1[0,1]$ and $t \in [0,1]$. This implies that
\begin{equation*}
T(0)(t) + c(T(f))(0)  + \int_{0}^{t} \beta(s) (T(f))'(\phi(s))ds = f(t).
\end{equation*}
Putting the values of $T(f)(0)$ and $(T(f))'(\phi(s))$, we get	
\begin{equation} \label{mainref1}
T(0)(t) + cT(0)(0) + c^2f(0) +  \int_{0}^{t} \Big[ \beta(s) T(0)'(\phi(s))  + \beta(s) \beta(\phi(s)) f'(\phi^2(s)) \Big] ds = f(t). 
\end{equation}
Choosing $f= 0$, we obtain	
\begin{equation*} \label{int1}
T(0)(t) + cT(0)(0) +  \int_{0}^{t} \beta(s) (T(0))'(\phi(s)) ds = 0.
\end{equation*}
Therefore, Equation \eqref{mainref1} reduces to	
\begin{equation}\label{ime1}
c^2f(0) +  \int_{0}^{t} \beta(s) \beta(\phi(s)) f'(\phi^2(s)) ds = f(t)
\end{equation}	
for every $f \in C^1[0,1]$ and $ t\in [0,1]$. Again, choosing $f = 1$, we get $c^2 =1$, and hence $c =\pm 1$.

Differentiating Equation \eqref{ime1}, we get 
\begin{equation*}
\beta(t) \beta(\phi(t)) f'(\phi^2(t)) = f'(t)
\end{equation*}
for every $f \in C^1[0,1]$ and $t \in [0,1]$. Now, take a function $f\in C^1[0,1]$ such that $f'(t)=1$ and $f'(\phi^2(t))= 0$, we get a contradiction. Hence, $\phi^2(t) =t$ for every $t \in [0,1]$. Moreover $\beta(t) \beta(\phi(t)) =1$ and $\phi^2(t) =t$ for every $t \in [0,1]$.
\end{proof}

\begin{thm}
Let $T$ be an isometric reflection on $C^1[0,1]$ of Form $(II)$. Then $\beta(t) \beta(\phi (t)) =1$ and $\phi^2(t) =t$ for every $t \in [0,1]$.
\end{thm}
\begin{proof}
	
An isometry on $C^1[0,1]$ of Form $(II)$ is described by
\begin{equation*} \label{Form2}
T(f)(t)= T(0)(t)+  c\overline{f(0)} + \int_{0}^{t} \beta(s)f'(\phi(s))ds
\end{equation*}
for every $f \in C^1[0,1]$ and $t \in [0,1]$. Thus, 
\begin{equation*}
T(f)(0) = T(0)(0) + c\overline{f(0)} \text{ and } (T(f))'(t)= (T(0))'(t)+ \beta(t) f'(\phi(t)).
\end{equation*}
If $T$ is an isometric reflection, then $T^2(f)(t)= f(t)$ for every $f \in C^1[0,1]$ and $t \in [0,1]$. It follows that	
\begin{equation*}
T(0)(t)+ c\overline{(T(f))(0)}  + \int_{0}^{t} \beta(s) (T(f))'(\phi(s))ds = f(t)
\end{equation*}
or	
\begin{equation} \label{mainref2}
T(0)(t)+ c\overline{ T(0)(0)} + \overline{f(0)} +   \int_{0}^{t} \Big[ \beta(s)  T(0)'(\phi(s))  + \beta(s)  \beta(\phi(s)) f'(\phi^2(s)) \Big] ds  = f(t). 
\end{equation}
Choosing $f =1$, we get
\begin{equation*} \label{ilt2}
T(0)(t)+ c\overline{ T(0)(0)} + \int_{0}^{t}  \beta(s)  T(0)'(\phi(s)) ds =0.
\end{equation*}
Thus, Equation \eqref{mainref2} reduces to
\begin{equation*} \label{ime2}
\overline{f(0)} +   \int_{0}^{t}  \beta(s)  \beta(\phi(s)) f'(\phi^2(s)) ds  = f(t). 
\end{equation*}
After differentiating this equation, we get	
\begin{equation}
\beta(t)  \beta(\phi(t)) f'(\phi^2(t)) ds  = f'(t) 
\end{equation}
for every $f \in C^1[0,1]$ and $t \in [0,1]$. Take a function $f\in C^1[0,1]$ such that $f'(t)=1$ and $f'(\phi^2(t))= 0$, this leads to a contradiction. It follows that $\beta(t) \beta(\phi(t)) =1$ and $\phi^2(t) =t$ for every $t \in [0,1]$.
\end{proof}

\begin{thm}
Let $T$ be an isometric reflection on $C^1[0,1]$ of Form $(III)$. Then $c = \pm 1$ with $\beta(t) \overline{\beta(\phi (t))} =1$ and $\phi^2(t) =t$ for every $t \in [0,1]$.
\end{thm}
\begin{proof}
An isometry $T$ on $C^1[0,1]$ of Form $(III)$ is described by
\begin{equation*} \label{Form3}
T(f)(t)= T(0)(t)+  cf(0) + \int_{0}^{t} \beta(s) \overline{f'(\phi(s))} ds
\end{equation*}
for every $f \in C^1[0,1]$ and $t \in [0,1]$. This implies that
\begin{equation*}
T(f)(0) = T(0)(0) + cf(0) \text{ and } (T(f))'(t)= (T(0))'(t)+ \beta(t) \overline{f'(\phi(t))}.
\end{equation*}
Since $T^2(f)(t)= f(t)$ for every $f \in C^1[0,1]$ and $t \in [0,1]$, we have	
\begin{equation*}
T(0)(t)+ c(T(f))(0)  + \int_{0}^{t} \beta(s) \overline{(T(f))'(\phi(s))}ds = f(t).
\end{equation*}
Thus
\begin{equation} \label{mainref3}
T(0)(t)+ cT(0)(0) + c^2f(0) +   \int_{0}^{t} \Big[ \beta(s)  \overline{T(0)'(\phi(s))}  + \beta(s) \overline{ \beta(\phi(s))} f'(\phi^2(s)) \Big] ds  = f(t). 
\end{equation}
Setting $f=0$,
\begin{equation*} \label{ilt3}
T(0)(t)+ cT(0)(0) + \int_{0}^{t}  \beta(s)  \overline{T(0)'(\phi(s))} ds =0.
\end{equation*}
So, Equation \eqref{mainref3} reduces to 
\begin{equation} \label{ime3}
c^2f(0) +   \int_{0}^{t}  \beta(s)  \overline{\beta(\phi(s))} f'(\phi^2(s)) ds  = f(t)
\end{equation}
for every $f \in C^1[0,1]$ and $t \in [0,1]$. Select $f=1$, we get $c^2 =1$, hence $c =\pm 1$.
	
Differentiating Equation \eqref{ime3}, we have
\begin{equation*}
\beta(t)  \overline{\beta(\phi(t))} f'(\phi^2(t)) = f'(t)
\end{equation*}
for every $f \in C^1[0,1]$ and $t \in [0,1]$. Again, choose $f\in C^1[0,1]$ satisfying $f'(t)=1$ and $f'(\phi^2(t))= 0$, this yields a contradiction. Consequently, $\beta(t) \overline{\beta(\phi(t))} =1$ and $\phi^2(t) =t$ for every $t \in [0,1]$.
\end{proof}
\begin{thm}
Let $T$ be an isometric reflection on $C^1[0,1]$ of Form $(IV)$. Then  $\beta(t) \overline{\beta(\phi (t))} =1$ and $\phi^2(t) =t$ for every $t \in [0,1]$.
\end{thm}
\begin{proof}
An isometry $T$ on $C^1[0,1]$ of Form $(IV)$ is described by 
\begin{equation*} \label{Form4}
T(f)(t)= T(0)(t)+  c \overline{f(0)} + \int_{0}^{t} \beta(s) \overline{f'(\phi(s))} ds
\end{equation*}
for every $f \in C^1[0,1]$ and $t \in [0,1]$. It follows that
\begin{equation*}
T(f)(0) = T(0)(0) + c\overline{f(0)} \text{ and } (T(f))'(t)= (T(0))'(t)+ \beta(t) \overline{f'(\phi(t))}.
\end{equation*}
Further, $T^2(f)(t)= f(t)$ gives
\begin{equation*}
T(0)(t)+ c \overline{(T(f))(0)}  + \int_{0}^{t} \beta(s) \overline{(T(f))'(\phi(s))}ds = f(t).
\end{equation*}
Thus,
\begin{equation} \label{mainref4}
T(0)(t)+ c \overline{T(0)(0)} + \overline{f(0)} +   \int_{0}^{t} \Big[ \beta(s)  \overline{T(0)'(\phi(s))}  + \beta(s) \overline{ \beta(\phi(s))} f'(\phi^2(s)) \Big] ds  = f(t). 
\end{equation}
Taking $f=0$, we get
\begin{equation*} \label{int4}
T(0)(t)+ c\overline{T(0)(0)} + \int_{0}^{t}  \beta(s)  \overline{T(0)'(\phi(s))} ds =0.
\end{equation*}
Therefore, Equation \eqref{mainref4} takes the form
\begin{equation} \label{ime4}
\overline{f(0)} +   \int_{0}^{t}  \beta(s)  \overline{\beta(\phi(s))} f'(\phi^2(s)) ds  = f(t) 
\end{equation}
for every $f \in C^1[0,1]$ and $t \in [0,1]$. Differentiating Equation \eqref{ime4}, 
\begin{equation*}
\beta(t)  \overline{\beta(\phi(t))} f'(\phi^2(t))  = f'(t) 
\end{equation*}
for every $f \in C^1[0,1]$ and $t \in [0,1]$. Choose $f\in C^1[0,1]$ such that $f'(t)=1$ and $f'(\phi^2(t))= 0$, this gives a contradiction. Therefore, $\beta(t) \overline{\beta(\phi(t))} =1$ and $\phi^2(t) =t$ for every $t \in [0,1]$.
\end{proof}

\section{Structures of Generalized bi-circular idempotents }

Now, we characterize generalized bi-circular idempotents on $C^1[0,1]$.
		
\begin{thm} \label{I_1}
The collection $\mathcal{C}= \{P_1, P_2\}$ is a family of  generalized bi-circular idempotents on $C^1[0,1]$ corresponding to an isometry $T$ of Form $(I)$ if and only if one of the following conditions is satisfied:
\begin{enumerate}
\item  Each $P_i$ is the average of the identity operator and isometric reflection with $\lambda_1+\lambda_2=0, c= \pm \lambda_1$ and $\beta(t) \beta(\phi(t)) = \pm \lambda_1, \phi^2(t) =t$ for every $t \in [0,1]$.
\item $\mathcal{C}$ is a family of bi-circular projections.
\end{enumerate}
\end{thm}
		
\begin{proof}
Let $\mathcal{C} =\{P_1,P_2\}$ be a family of generalized bi-circular idempotents on $C^1[0,1]$ corresponding to an isometry $T$ of Form $(I)$. Then $\lambda_1P_1(f)(t)+\lambda_2P_2(f)(t) = T(f)(t),$ where $T$ is a surjective isometry and $\lambda_1, \lambda_2$ are unit modulus complex numbers. 

By Proposition \ref{T0=0}, it follows that $T(0)=0$. Consequently, $T$ takes the form			
\begin{equation*}
T(f)(t)= cf(0) + \int_{0}^{t} \beta(s)f'(\phi(s))ds 
\end{equation*}	
for every $f \in C^1[0,1]$ and $t \in [0,1]$. Here, we can see that $T$ is a surjective linear isometry. Therefore, $P_1$ and $P_2$ are projections. Thus, we have the identity
\begin{equation} \label{GBPidentity}
T^2(f)(t) - (\lambda_1 +\lambda_2) T(f)(t) + \lambda_1\lambda_2 f(t) = 0  
\end{equation}			
for every $f \in C^1[0,1]$ and $t \in [0,1]$. Also,
\begin{equation} \label{p}
P_i = \frac{T-\lambda_jI}{\lambda_i-\lambda_j}; i,j = 1,2, i \neq j.
\end{equation}		
Now substituting the Form of $T$ in the Equation \eqref{GBPidentity}, we get 
\begin{equation} \label{main1}
\Bigg[c^2f(0) + \int_{0}^{t} \beta(s) \beta(\phi(s)) f'(\phi^2(s))ds\Bigg] - (\lambda_1 +\lambda_2) \Bigg[ cf(0) + \int_{0}^{t} \beta(s)f'(\phi(s))ds \Bigg] + \lambda_1\lambda_2 f(t) =0.
\end{equation}
We choose $f= 1$ in the Equation \eqref{main1}, which gives
\begin{align*}
c^2 - (\lambda_1 +\lambda_2)c + \lambda_1 \lambda_2 =0 \implies (c-\lambda_1)(c- \lambda_2) =0 \implies  c =\lambda_1,\lambda_2.
\end{align*}
Differentiating the Equation \eqref{main1} on both sides, we have 
\begin{equation} \label{me1}
\beta(t)\beta(\phi(t)) f'(\phi^2(t)) - (\lambda_1 +\lambda_2) \beta(t) f'(\phi(t)) + \lambda_1\lambda_2 f'(t) =0
\end{equation}
for every $f \in C^1[0,1]$ and $t \in [0,1]$. 

We assert that $\phi^2(t) =t$ for every $t \in [0,1]$. Clearly, if $\phi(t) =t$, then $\phi^2(t) = t$. Suppose $\phi(t) \neq t$ for some $t \in [0,1]$. In this case, $t, \phi(t)$, and $\phi^2(t)$ cannot be all distinct. If they are, we choose a function $f \in C^1[0,1]$ such that $f'(t) =1$ and $f'(\phi(t)) = f'(\phi^2(t))=0$. This implies $\lambda_1\lambda_2=0$, which is a contradiction. Hence, we conclude that $\phi^2(t)= t$ for every $t \in [0,1]$.

If $\phi(t) \neq t$ for some $t \in [0,1]$, then we select a function $f \in C^1[0,1]$ such that $f'(t) =1 $ and $f'(\phi(t))= 0$. From this setup, we get $(\lambda_1 + \lambda_2)\beta(t) =0$. Since $\beta(t)$ is unimodular, $\lambda_1 + \lambda_2 = 0$. Consequently, $\beta(t)\beta(\phi(t)) = \pm \lambda_1$ for every $t \in [0,1]$.

So, Equation \ref{p} reduces to
\begin{equation*} 
P_i= \frac{T- \lambda_j I}{\lambda_i - \lambda_j} = \frac{T + \lambda_iI}{2\lambda_i} = \frac{I + S}{2}
\end{equation*}
where, $S= \frac{1}{\lambda_i}T$. It is clear that $S^2 =I$. This proves the first assertion.	

If $\phi(t) =t$ for every $t \in [0,1]$, then $\beta^2(t) f'(t) - (\lambda_1 + \lambda_2) \beta(t) f'(t) + \lambda_1\lambda_2 f'(t) =0$. Moreover, $\beta^2(t) - (\lambda_1 + \lambda_2) \beta(t) + \lambda_1 \lambda_2 =0$. It follows that $\beta(t) = \lambda_1$ or $\lambda_2$. Since $[0,1]$ is a connected set, $\beta(t)$ must be a constant function.

Now, if $\beta(t) =\lambda_1$ for all $t \in [0,1]$, then $ T(f)(t) = cf(0) + \lambda_1 (f(t) - f(0))$. Therefore,
\begin{equation*} 
P_i(f)(t) = \frac{(c- \lambda_1) f(0) + (\lambda_1 - \lambda_j)f(t)}{\lambda_i-\lambda_j}; i,j = 1,2, i \neq j.
\end{equation*} 
Hence for all $\alpha, \beta$ of unimodular complex numbers,  $f \in C^1[0,1]$ and $t \in [0,1]$, we have
\begin{align*}
\alpha P_1(f)(t) + \beta P_2(f)(t) &= \alpha \frac{(c- \alpha) f(0) + (\alpha - \beta)f(t)}{\alpha - \beta} + \beta \frac{(c- \alpha) f(0)}{\beta -\alpha} \\
&= cf(0) + \alpha (f(t) - f(0)).
\end{align*}
This is a surjective isometry on $C^1[0,1]$. Therefore, the collection $\mathcal{C}$ forms a family of bi-circular idempotents.

Again, if $\beta(t) =\lambda_2$ for all $t \in [0,1]$, then $ T(f)(t) = cf(0) + \lambda_2 (f(t) - f(0))$. Moreover,
\begin{equation*} 
P_i(f)(t) = \frac{(c- \lambda_2) f(0) + (\lambda_2 - \lambda_j) f(t)}{\lambda_i-\lambda_j}; i,j = 1,2, i \neq j.
\end{equation*}
Similarly, one can verify that $\mathcal{C}$ forms a family of bi-circular idempotents. This completes the proof.
\end{proof}
		
\begin{thm} \label{I_2}
The collection $\mathcal{C} = \{P_1, P_2\}$ is a family of generalized bi-circular idempotents on $C^1[0,1]$ corresponding to an isometry $T$ of Form $(II)$ if and only if one of the following conditions is satisfied:
\begin{enumerate}
\item Each $P_i$ is the average of the identity operator and isometric reflection with $\lambda_1+\lambda_2=0$ and $\beta(t) \beta(\phi(t)) = \pm \lambda_1,\  \phi^2(t) =t$ for every $t \in [0,1]$.
\item $\mathcal{C}$ is a family of bi-circular idempotents.
\end{enumerate}
\end{thm}
\begin{proof}
Let $\mathcal{C} = \{P_1,P_2\}$ be a family of generalized bi-circular idempotents on $C^1[0,1]$ corresponding to an isometry $T$ of Form $(II)$.  Then $\lambda_1P_1(f)(t) +\lambda_2P_2(f)(t) = T(f)(t),$ where $T$ is a surjective isometry for some complex numbers $\lambda_1, \lambda_2$ of unit modulus. Proposition \ref{T0=0} gives $T(0) =0$, thus $T$ takes the form
\begin{equation*}
T(f)(t)= c\overline{f(0)} + \int_{0}^{t} \beta(s)f'(\phi(s))ds
\end{equation*}
for every $f \in C^1[0,1]$ and $t \in [0,1]$. Idempotent $P_i$ reduces to
\begin{equation} \label{p_2}
P_i(f)(t) = \frac{1}{\lambda_i-\lambda_j} \Big[ c\overline{f(0)} + \int_{0}^{t} \beta(s)f'(\phi(s))ds - \lambda_j f(t)\Big].
\end{equation}
On differentiating for $i=1$, we obtain
\begin{equation*}
(P_1(f))'(t) = \frac{1}{\lambda_1-\lambda_2} \Big[ \beta(t)f'(\phi(t)) - \lambda_2 f'(t)   \Big] .
\end{equation*}			
Taking $t=0$ in the Equation $\eqref{p_2}$, we have
\begin{equation*}
P_1(f)(0) = \frac{1}{\lambda_1-\lambda_2} \Big[ c\overline{f(0)} - \lambda_2 f(0)\Big].
\end{equation*}
Since $P_i; i = 1,2$ is a generalized bicircular idempotent, we have 
\begin{equation*}
TP_i(f)(t)= \lambda_iP_i(f)(t)
\end{equation*} 
for every $f \in C^1[0,1]$ and $t \in [0,1]$. This implies that
\begin{align*}
c\overline{(P_1(f))(0)} + \int_{0}^{t}\beta(s)(P_1(f))'(\phi(s))ds = \frac{\lambda_1}{\lambda_1-\lambda_2} \Big[ c\overline{f(0)} +\int_{0}^{t} \beta(s)f'(\phi(s)))ds - \lambda_2 f(t)\Big].
\end{align*}	
Substituting the values of $P_1(f)(0)$ and $(P_1(f))'(\phi(s))$, we have
\begin{align*}
\frac{c}{\overline{\lambda_1}-\overline{\lambda_2}}\Big[\overline{ c\overline{f(0)} -\lambda_2 f(0)}\Big] 
&+  \frac{1}{\lambda_1-\lambda_2}  \int_{0}^{t}\beta(s) \Big[ \beta(\phi(s))f'(\phi^2(s)) - \lambda_2 f'(\phi(s)) \Big] ds  \\
&= \frac{\lambda_1}{\lambda_1-\lambda_2} \Big[ c\overline{f(0)} +\int_{0}^{t} \beta(s) f'(\phi(s)) ds - \lambda_2 f(t)\Big].
\end{align*}
After simplification, we get
\begin{equation*} \label{le2}
\frac{1}{\overline{\lambda_1}-\overline{\lambda_2}} f(0) + \frac{1}{\lambda_1-\lambda_2} \int_{0}^{t} \Big[\beta(s)\beta(\phi(s)) f'(\phi^2(s)) - (\lambda_1 + \lambda_2) \beta(s) f'(\phi(s))\Big] ds + \frac{\lambda_1 \lambda_2 }{\lambda_1 - \lambda_2}f(z) =0.
\end{equation*}
Differentiating on both sides, we have 
\begin{equation} \label{me2}
\beta(t)\beta(\phi(t))f'(\phi^2(t)) - (\lambda_1 + \lambda_2) \beta(t) f'(\phi(t))  + \lambda_1 \lambda_2 f'(t) =0
\end{equation}		
for every $f \in C^1[0,1]$ and $t \in [0,1]$. 

Now, proceed in the same manner as in Theorem \ref{I_1}. We conclude that $\phi^2(t)=t$ for every $t \in [0,1]$.

If $\phi(t) \neq t$ for some $t \in [0,1]$, then $\lambda_1+ \lambda_2=0$. Hence, $\beta(t) \beta(\phi(t)) = \pm \lambda_1$ for every $t \in [0,1]$.

Thus, Equation \eqref{p_2} gives
\begin{equation*} 
P_i = \frac{T - \lambda_j I}{\lambda_i - \lambda_j} = \frac{T + \lambda_i I}{2 \lambda_i}= \frac{I + S}{2}
\end{equation*}
where $S = \frac{1}{\lambda_i}T$ with $S^2 = I$. This proves the first statement.

If $\phi(t) = t$ for every $t \in [0,1]$, then we get $\beta(t) = \lambda_1$ or $\lambda_2$. Since $[0,1]$ is a connected set, $\beta(t)$ must be a constant function.

Thus, $\beta(t) =\lambda_1$ for all $t \in [0,1]$, and hence $Tf(t) = c \overline{f(0)}  + \lambda_1 (f(t)- f(0))$. Consequently,
\begin{equation*} 
P_i(f)(t) = \frac{c \overline{f(0)} - \lambda_1f(0) + (\lambda_1 - \lambda_j)f(t) }{\lambda_i-\lambda_j}; i,j = 1,2, i \neq j.
\end{equation*} 
Proceeding as in Theorem \ref{I_1}, we obtain the desired conclusion.

Again, if $\beta(t) =\lambda_2$ for all $t \in [0,1]$, then $Tf(t) = c \overline{f(0)}  + \lambda_2(f(t) - f(0))$. Thus,
\begin{equation*} 
P_i(f)(t) = \frac{c \overline{f(0)} - \lambda_2f(0) + (\lambda_2 - \lambda_j)f(t) }{\lambda_i-\lambda_j}; i,j = 1,2, i \neq j.
\end{equation*}
Similarly, the collection $\mathcal{C} = \{P_1, P_2\}$ is a family of bi-circular idempotents. This completes the proof.
\end{proof}	

\begin{thm} \label{I_3}
The collection  $\mathcal{C}= \{P_1,P_2\}$ is a family of generalized bi-circular idempotent on $C^1[0,1]$ corresponding to an isometry $T$ of form $(III)$ if and only if  $\mathcal{C}$ is a family of bi-circular idempotents with $c=\lambda_1, \lambda_2$ and $\beta(t) \overline{\beta(\phi(t))} = 1, \phi^2(t) =t$ for every $t \in [0,1]$.
\end{thm}	
\begin{proof}
Let $\mathcal{C} =\{P_1,P_2\}$ be a family of generalized bi-circular idempotents on $C^1[0,1]$ corresponding to an isometry $T$ of Form $(III)$.  Then, for some unimodular complex numbers $\lambda_1,\lambda_2$, We have
$$
\lambda_1 P_1(f)(t) + \lambda_2P_2(f)(t) = T(f)(t),
$$
where $T$ is a surjective isometry. Using Proposition \ref{T0=0}. The isometry $T$ reduces to the form
\begin{equation*}
T(f)(t)=  c f(0) + \int_{0}^{t}\beta(s) \overline{f'(\phi(s))}ds  
\end{equation*}
for every $f \in C^1[0,1]$ and $t \in [0,1]$. Now, from Equation $\eqref{p}$, we obtain
\begin{equation} \label{p_3}
P_i(f)(t) = \frac{T(f)(t)- \lambda_j f(t)}{\lambda_i-\lambda_j} = \frac{1}{\lambda_i-\lambda_j} \Bigg[cf(0) + \int_{0}^{t} \beta(s)\overline{f'(\phi(s))}ds - \lambda_j f(t)\Bigg].
\end{equation}
On differentiating both sides of the Equation \eqref{p_3} for $i=1$, we have
\begin{equation*}
(P_1(f))'(t) = \frac{1}{\lambda_1-\lambda_2} \Big[\beta(t)\overline{f'(\phi(t))} - \lambda_2 f'(t)\Big].
\end{equation*}
Also put $t =0$ in \eqref{p_3} for $i=1$,
\begin{equation*}
P_1(f)(0) = \frac{1}{\lambda_1-\lambda_2} \Big[cf(0) - \lambda_2 f(0)\Big].
\end{equation*}
Now from the result $TP_1(f)(t) = \lambda_1P_1(t)$, we have
\begin{equation*}
c P_1(f)(0) + \int_{0}^{t}\beta(s) \overline{(P_1(f))'(\phi(s))}ds  = \frac{\lambda_1}{\lambda_1-\lambda_2} \Big[cf(0) + \int_{0}^{t} \beta(s)\overline{f'(\phi(s))}ds - \lambda_2 f(t)\Big].
\end{equation*}
Substituting the values of $P_1(f)(0)$ and $(P_1(f))'(t)$ 
\begin{align*}
\frac{c}{\lambda_1-\lambda_2} \Big[cf(0) - \lambda_2 f(0)\Big] &+ \frac{1}{\overline{\lambda_1} - \overline{\lambda_2}}\int_{0}^{t}\beta(s) \overline{\Big[\beta(\phi(s))\overline{f'(\phi^2(s))} - \lambda_2 f'(\phi(s))\Big]}ds \\
&= \frac{\lambda_1}{\lambda_1-\lambda_2} \Big[ cf(0) + \int_{0}^{t} \beta(s)\overline{f'(\phi(s))}ds - \lambda_2 f(t)\Big].
\end{align*}				
After simplification, we get the following equation.
\begin{equation} \label{le3}
\frac{1}{\lambda_1-\lambda_2} \big[(c^2-c\lambda_1 -c\lambda_2) f(0) + \lambda_1  \lambda_2 f(t)\big] + \frac{1}{\overline{\lambda_1} - \overline{\lambda_2}} \int_{0}^{t} \beta(s)\overline{\beta(\phi(s))} f'(\phi^2(s))ds=0.
\end{equation}			
Selecting $f= 1$, we obtain
\begin{align*}
c^2-c\lambda_1 -c\lambda_2 + \lambda_1  \lambda_2 = 0 \implies c = \lambda_1, \lambda_2.
\end{align*}
Differentiating the Equation \eqref{le3}, then	
\begin{equation}
\frac{\lambda_1\lambda_2}{\lambda_1-\lambda_2} f'(t) + \frac{1}{\overline{\lambda_1} - \overline{\lambda_2}} \beta(t)\overline{\beta(\phi(t))} f'(\phi^2(t)) = 0
\end{equation}
for every $f \in C^1[0,1]$. This implies that $\beta(t)\overline{\beta(\phi(t))} =1 $ and $\phi^2(t) = t$, for every $t \in [0,1]$. 

We started the proof by considering $\lambda_1 P_1 + \lambda_2 P_2$ to be an isometry of Form $(III)$. From this we deduce that $\beta(t)\overline{\beta(\phi(t))} =1 $ and $\phi^2(t) = t$ for every $t \in [0,1]$, which guarantees that $P_i = \frac{T- \lambda_jI}{\lambda_i - \lambda_j}$ is an idempotent map. Hence the identity $\alpha P_1 + \beta P_2 =T$ holds without any condition on scalar $\alpha, \beta$. Thus, the collection $\mathcal{C}$ forms a family of bi-circular idempotents. This concludes the proof. 
\end{proof}

\begin{thm} \label{I_4}
The collection $\mathcal{C} = \{P_1,P_2\}$ is a family of generalized bi-circular idempotents on $C^1[0,1]$ corresponding to an isometry $T$ of Form $(IV)$ if and only if  $\mathcal{C}$ is a family of bi-circular idempotents with $\beta(t) \overline{\beta(\phi(t))} = 1, \phi^2(t) =t$ for every $t \in [0,1]$.
\end{thm}
		
\begin{proof}
Let $\mathcal{C} = \{P_1,P_2\}$ be a family generalized bi-circular idempotent on $C^1[0,1]$ corresponding to an isometry $T$ of Form $(IV)$. Then 
$$
\lambda_1P_1(f)(t)+\lambda_2P_2(f)(t) = T(f)(t),
$$
where $T$ is a surjective isometry for some complex numbers $\lambda_1, \lambda_2$ of unit modulus. Using the result that $T(0) =0$. Then $T$ takes the form
\begin{equation}
T(f)(t)= c\overline{f(0)} + \int_{0}^{t}\beta(s) \overline{f'(\phi(s))}ds
\end{equation}
for every $f \in C^1[0,1]$ and $t \in [0,1]$. Now from the Equation $\eqref{p}$, we get			
\begin{equation} \label{p_4}
P_i(f)(t) = \frac{T(f)(t)- \lambda_j f(t)}{\lambda_i-\lambda_j} = \frac{1}{\lambda_i-\lambda_j} \Bigg[ c\overline{f(0)}  + \int_{0}^{t} \beta(s)\overline{f'(\phi(s))}ds - \lambda_j f(t)\Bigg]; i,j =1,2, i\neq j.
\end{equation}
On differentiating Equation \eqref{p_4} for $i=1$,
\begin{equation*}
(P_1(f))'(t) = \frac{1}{\lambda_1-\lambda_2} \Big[ \beta(t)\overline{f'(\phi(t))} - \lambda_2 f'(t)\Big].
\end{equation*}
Also put $t =0$ in \eqref{p_4} for $i=1$,
\begin{equation*}
P_1(f)(0) = \frac{1}{\lambda_1-\lambda_2} \Big[ c\overline{f(0)} - \lambda_2 f(0)\Big].
\end{equation*}
Now from the result $TP_1(f)(t) = \lambda_1P_1(t)$, we have
\begin{equation*}
c \overline{P_1(f)(0)} + \int_{0}^{t}\beta(s) \overline{(P_1(f))'(\phi(s))}ds  = \frac{\lambda_1}{\lambda_1-\lambda_2} \Big[ c\overline{f(0)} + \int_{0}^{t} \beta(s)\overline{f'(\phi(s))}ds - \lambda_2 f(t)\Big].
\end{equation*}
Putting the values of $P_1(f)(0)$ and $(P_1(f))'(t)$ 
\begin{align*}
\frac{c}{\overline{\lambda_1}- \overline{\lambda_2}}  \overline{\Big[ c \overline{f(0)} - \lambda_2 f(0)\Big]} &+ \frac{1}{\overline{\lambda_1}- \overline{\lambda_2}}\int_{0}^{t}\beta(s) \overline{\Big[ \beta(\phi(s)) \overline{f'(\phi^2(s))} - \lambda_2 f'(\phi(s))\Big]}ds \\
&= \frac{\lambda_1}{\lambda_1-\lambda_2} \Big[c \overline{f(0)} + \int_{0}^{t} \beta(s)\overline{f'(\phi(s))}ds - \lambda_2 f(t)\Big].
\end{align*}	
or
\begin{equation*} \label{le4}
\frac{1}{\overline{\lambda_1}- \overline{\lambda_2}} f(0) + \frac{\lambda_1\lambda_2}{\lambda_1-\lambda_2} f(t) + \frac{1}{\overline{\lambda_1}- \overline{\lambda_2}}\int_{0}^{t}\beta(s) \overline{\beta(\phi(s))} f'(\phi^2(s))ds  = 0.
\end{equation*}
On differentiating, we get
\begin{equation} 
\frac{\lambda_1\lambda_2}{\lambda_1-\lambda_2} f'(t) + \frac{1}{\overline{\lambda_1}- \overline{\lambda_2}} \beta(t) \overline{\beta(\phi(t))} f'(\phi^2(t))  = 0
\end{equation}
for every $f \in C^1[0,1]$. This gives us $\beta(t) \overline{\beta(\phi(t))}= 1$ and $\phi^2(t) =t$, for every $t \in [0,1]$. 

Applying the arguments used in the final step of the proof of Theorem \ref{I_3}, it follows that the collection $\mathcal{C}$ constitutes a family of bi-circular idempotents. This completes the proof.
\end{proof}

\end{document}